\documentclass{article}
\usepackage{mathrsfs}
\usepackage{a4wide}
\usepackage{amsmath,amsfonts,amssymb,amsthm,dsfont}

\usepackage[active]{srcltx}
\DeclareMathAlphabet{\mathpzc}{OT1}{pzc}{m}{it}
\def \beq{\begin{equation}}
\def \eeq{\end{equation}}

\def\and {{\rm \; and \;}}

\newcommand{\R}{{\mathbb R}}

\newcommand{\y}{{\bf y}}

\newcommand{\x}{{\bf x}}

\newtheorem{theorem}{Theorem}[section]

\newtheorem{proposition}[theorem]{Proposition}

\newtheorem{lemma}[theorem]{Lemma}

\theoremstyle{definition}

\newcommand{\eh}{\hfill}
\newlength{\sperrT}

\numberwithin{equation}{section}

\begin{document}

\noindent 
\begin{center}
\textbf{\large On the regularity of the Hausdorff distance between
  spectra of perturbed magnetic Hamiltonians}
\end{center}

\begin{center}
August 30, 2011
\end{center}

\vspace{0.5cm}

\begin{center}
\textbf{ 
Horia D. Cornean\footnote{Department of Mathematical Sciences,
  Aalborg University, Fredrik Bajers Vej 7G, 9220 Aalborg, Denmark},
Radu Purice\footnote{Institute of Mathematics 'Simion Stoilow' of the
  Romanian Academy, P. O. Box 1-764, RO-014700 Bucharest, Romania}}

\end{center}

\begin{abstract}
We study the regularity properties of the Hausdorff distance between
spectra of continuous 
Harper-like operators. As a special case we obtain H\"{o}lder continuity of
this Hausdorff distance with respect to the intensity of the magnetic
field for a large class of magnetic elliptic 
(pseudo)differential operators with long range magnetic fields.
\end{abstract}

\noindent

\section{Introduction} 

Analytic perturbation theory tells us that if $V$ is relatively bounded to
$H_0$, then the spectrum of $H_\lambda=H_0+\lambda V$ is at a Hausdorff distance
of order $|\lambda|$ from the spectrum of $H_0$. This property is 
not true for singular perturbations (like for example the magnetic perturbation
coming from a constant field), neither in the discrete nor in the 
continuous case.

Maybe the first proof of spectral stability of discrete 
Harper operators with respect to the variation of the intensity $b\geq
0$ of the external magnetic field is due to Elliott \cite{Ell}. The result is
refined in \cite{BEY} where it is shown that the gap boundaries are
$\frac{1}{3}$-H\"older continuous in $b$. Later results by Avron, van
Mouche and Simon \cite{AMS}, Helffer and Sj\"ostrand \cite{He-Sj1,
  He-Sj2}, and Haagerup and R{\o}rdam \cite{HR} 
pushed the exponent up to $\frac{1}{2}$. In fact they prove
more, they show that the Hausdorff distance between spectra behaves like
$|b-b_0|^{\frac{1}{2}}$. These results are
optimal in the sense that the H\"older constant is independent of the
length of the eventual gaps, and it  is known that these gaps can close down
precisely like $|b-b_0|^{\frac{1}{2}}$ 
near rational values of $b_0$ \cite{He-Sj2, HKS}. Note that Nenciu
\cite{Nen3} proves a similar result for a much larger class of
discrete Harper-like operators. Many other spectral properties of
Harper operators can be found in a paper by Herrmann and Janssen \cite{HJ}. 

In the continuous case, the stability of gaps for Schr\"odinger
operators was first shown by Avron
and Simon \cite{AS}, and Nenciu \cite{Nen2}. In \cite{If} a very general 
result is obtained for perturbations of the anisotropic Laplacean. In \cite{AMP} spectral 
continuity is proven for a large class of Hamiltonians defined by elliptic symbols. Nenciu's result implicitly
gives a $\frac{1}{2}$-H\"older continuity in $b$ for the Hausdorff
distance between spectra.  Then in \cite{BC} the H\"older exponent of
gap edges was pushed up to $\frac{2}{3}$.   

The first proof of Lipschitz continuity of gap edges 
for discrete Harper-like operators 
was given by Bellissard \cite{Bell} (later on Kotani
\cite{Ko} extended his method to more general regular lattices and
dimensions larger than two). Very recently a completely different proof was given
in \cite{Cornean}. 

Our main technical result in this paper is Theorem \ref{teorema1}, extending a previous result of Nenciu \cite{Nen3} and 
asserting H\"{o}lder continuity of a specific order for a class of bounded self-adjoint
operators having a locally integrable integral kernel satisfying a
weighted Schur-Holmgren estimate \eqref{novem1}. This result, combined
with the magnetic quantization \cite{KO,MP1,MPR2} and the associated
magnetic pseudodifferential calculus developped in
\cite{MPR1,IMP1,IMP2,LMS}, allow us to we prove Theorem \ref{teorema2} 
stating H\"{o}lder continuity of order $1/2$ of the spectrum of
resolvents associated to a large class of elliptic Hamiltonians in a $BC^\infty$
magnetic field, with respect to the intensity of the magnetic
field. The case of unbounded operators will be considered elsewhere.

\subsection{The setting and the main result}
 Consider the Hilbert space $L^2(\mathbb{R}^d)$ with $d\geq 2$. Let
 $\langle x\rangle:=\sqrt{1+|\x|^2}$ and let $\alpha\geq 0$. 
We consider bounded integral operators 
$T\in B(L^2(\mathbb{R}^d))$ to which we can associate a locally integrable kernel 
$T(\x,\x')$ which is continuous outside the diagonal and obeys the
following weighted Schur-Holmgren estimate:
\begin{equation}\label{novem1}
||T||_{1,\alpha}:= \max\left\{ \sup_{\x'\in
    \R^d}\int_{\R^d}|T(\x,\x')|\langle \x-\x'\rangle^\alpha d\x, \;\sup_{\x\in
    \R^d}\int_{\R^d}|T(\x,\x')|\langle \x-\x'\rangle^\alpha d\x'\right\}<\infty.
\end{equation}
Let us denote the set of all these operators with
$\mathcal{C}_{1,\alpha}$. When $\alpha=0$, we need to introduce a
uniformity condition. 
Let $\chi$ be the characteristic function of the interval $[0,1]$ and
define 
\begin{equation}\label{chichi}
\mathbb{\R}^d\times \mathbb{\R}^d\ni (\x,\x')\mapsto 
\chi_M(\x,\x'):=\chi(|\x-\x'|/M),\quad M\geq 1.
\end{equation}
If $T\in \mathcal{C}_{1,0}$ we denote by $T_M$ the operator given by the
integral kernel $\chi_M(\x,\x')T(\x,\x')$. 
Then we define $\mathcal{C}_{\rm unif}$ to be the subset of $\mathcal{C}_{1,0}$
consisting of operators obeying the estimate 
\begin{equation}\label{chichi2}
\lim_{M\to \infty}||T-T_M||_{1,0}=0.
\end{equation}
Note that if we only consider kernels $T(\x,\x')$ which are dominated
by $L^1$ functions of $\x-\x'$, then 
$\mathcal{C}_{\rm unif}=\mathcal{C}_{1,0}$.

For $T\in \mathcal{C}_{1,\alpha}$, we are interested in a family of Harper-like operators 
$\{T_b\}_{b\in\R}$ given by kernels of the form $e^{ib
  \varphi(\x,\x')}T(\x,\x')$ with $\varphi: \mathbb{\R}^d\times
\mathbb{\R}^d\mapsto \mathbb{R}$ a continuous phase function satisfying the two properties:
\begin{align}\label{iunie1}
\varphi({\bf x},{\bf x}')=-\varphi(\x',\x)\quad {\rm and}\quad 
|\varphi({\bf x},\y)+\varphi(\y,\x')-\varphi(\x,\x')|\leq |\x-\y|\; |\y-\x'|.
\end{align}
Clearly, $\{T_b\}_{b\in\R}\subset 
\mathcal{C}_{1,\alpha}$. 

The Hausdorff distance between two real compact sets $A$ and $B$ is 
defined as:
\begin{align}\label{kiki1}
d_H(A,B):=\max\left \{\sup_{x\in A}\inf_{y\in B}|x-y|,
\; \sup_{y\in B}\inf_{x\in A}|x-y|\right \}.
\end{align}

And here is our main technical result:
\begin{theorem}\label{teorema1} Let $H$ be self-adjoint and consider 
a family of Harper-like operators $\{H_b\}_{b\in\R}$ as above. Then the map 
$$\R\ni b\mapsto d_H(\sigma(H_{b}),\sigma(H))\in \R_+$$ is continuous if 
$H\in \mathcal{C}_{\rm unif}$. Moreover, if $H\in
\mathcal{C}_{1,\alpha}$ with $\alpha>0$, then the above map is
H\"older continuous with exponent $\beta:=\min\{1/2,\alpha/2\}$. More
precisely, for all $b_0$ we can find a constant $C>0$ such that:
 \begin{align}\label{kiki2}
d_H(\sigma(H_{b_0+\delta}),\sigma(H_{b_0}))\leq C\;|\delta|^{\beta}.
\end{align}
\end{theorem}

\vspace{0.5cm}

{\bf Remark 1}. Denoting by $\delta =b-b_0$, then according to our notations 
we have that $H_b=\left (H_{b_0}\right )_{\delta}$. It means that it
is enough to prove the theorem at $b_0=0$.

{\bf Remark 2}. It is natural to ask if the condition $H\in
\mathcal{C}_{1,\alpha}$ is 
optimal in order to insure a H\"older continuity of order $\min\{1/2,\alpha/2\}$; we believe in any case that 
if $\alpha$ becomes smaller and smaller, one cannot expect the H\"older coefficient to remain $1/2$. Similarly, if $\alpha=0$ 
it is unlikely to expect more than continuity of the Hausdorff distance. 

\section{Proof of Theorem \ref{teorema1}}
 
Let $g\in C_0^\infty(\R^d)$ with $0\leq g\leq 1$, $g(\x)=1$ if
$|\x|\leq 1/2$ and $g(\x)=0$ if $|\x|\geq 2$. If $\y\in\R^d$, denote
by $g_{\y}(\x)=g(\x-\y)$. By standard arguments, we may assume that 
$\sum_{\gamma\in \mathbb{Z}^d}g_{\gamma}^2(\x)=1$ for all
$\x\in\R^d$. For each $g_\gamma$ there is a finite number of neighbors
whose supports are not disjoint from ${\rm supp}(g_\gamma)$, uniformly
in $\gamma$. 

Denote by $g_{\y,b}(\x):=g_{\y}(b^{1/2}\x)=g(b^{1/2}\x-\y)$. In this
way we constructed a
locally finite, quadratic partition of unity obeying 
\begin{equation}\label{kiki3}
\sum_{\gamma\in \mathbb{Z}^d}g_{\gamma,b}^2(\x)=1,\quad \x\in\R^d,
\end{equation}
and if $V_{\gamma,b}$ denotes the set of functions
$g_{\gamma',b}$ whose supports are not disjoint from the support of 
$g_{\gamma,b}$, then $\sup_{\gamma\in \mathbb{Z}^d}\; \#\{V_{\gamma,b}\}$
is independent of $b$. Moreover, if $\chi_{\gamma,b}$ is 
the characteristic function of the support of $g_{\gamma,b}$ we have: 
\begin{align}\label{kiki4}
&{\rm supp}(g_{\gamma,b})\subset \{\x\in\R^d:\;
|\x-b^{-1/2}\gamma|\leq 2 b^{-1/2}\},\\ 
\label{kiki5}
&|g_{\gamma,b}(\x)-g_{\gamma,b}(\y)|\leq ||\;|\nabla g|\;||_\infty^{\epsilon}\;
b^{\epsilon/2}|\x-\y|^\epsilon\;\{\chi_{\gamma,b}(\x)+
\chi_{\gamma,b}(\y)\},\quad 0\leq
\epsilon\leq 1. 
\end{align}

\begin{lemma}\label{lemma1}
Let $\{T_\gamma\}_{\gamma\in\mathbb{Z}^d}\subset B(L^2(\R^d))$
possibly depending on $b$ such
that 
\begin{equation}\label{kiki6}
|||T|||_\infty:=\sup_{\gamma\in\mathbb{Z}^d}||T_\gamma||<\infty.
\end{equation}
Define on compactly supported functions the
maps 
$$\psi\mapsto \Gamma(T)(\psi):=\sum_{\gamma\in\mathbb{Z}^d}\chi_{\gamma,b}\;
T_\gamma\; 
\chi_{\gamma,b}\psi,\quad \tilde{\Gamma}(T)(\psi):=
\sum_{\gamma\in\mathbb{Z}^d}\chi_{\gamma,b}\;
\left\vert T_\gamma\; 
\chi_{\gamma,b}\psi\right\vert.$$ 
Then both $\Gamma(T)$ and $\tilde{\Gamma}(T)$ can be extended by
continuity to
 bounded maps on $L^2(\R^d)$ and 
there exists a constant $C$ independent of
$b$ such that $\max\{||\tilde{\Gamma}(T)||,\;||\Gamma(T)||\}\leq C\; |||T|||_\infty$. 
\end{lemma}
\begin{proof} Let $\psi\in L^2(\R^d)$ with compact
  support. We have:
 \begin{align}\label{kiki7}
||\Gamma(T)(\psi)||^2&\leq \sum_{\gamma\in\mathbb{Z}^d}
\sum_{\gamma'\in V_{\gamma,b}} |\left \langle  \chi_{\gamma',b}\;
T_{\gamma'}\; 
\chi_{\gamma',b}\psi ,\chi_{\gamma,b}\;
T_\gamma\; 
\chi_{\gamma,b}\psi\right \rangle| \nonumber \\
&\leq \sum_{\gamma\in\mathbb{Z}^d}\sum_{\gamma'\in V_{\gamma,b}}||T_{\gamma'}\; 
\chi_{\gamma',b}\psi||\; ||T_{\gamma}\; 
\chi_{\gamma,b}\psi||\nonumber \\
&\leq \frac{|||T|||_\infty^2}{2}
\sum_{\gamma\in\mathbb{Z}^d}\sum_{\gamma'\in V_{\gamma,b}}\left (
||\chi_{\gamma',b}\psi||^2+
||\chi_{\gamma,b}\psi||^2\right)\leq C\;|||T|||_\infty^2||\psi||^2, 
\end{align}
where in the last inequality we used:
$$\sum_{\gamma\in\mathbb{Z}^d}\sum_{\gamma'\in V_{\gamma,b}}
||\chi_{\gamma',b}\psi||^2=\int_{\R^d}|\psi(\x)|^2\left \{\sum_{\gamma\in\mathbb{Z}^d}
\sum_{\gamma'\in
  V_{\gamma,b}}\chi_{\gamma',b}(\x)\right\}d\x\leq
C||\psi||^2.$$ 
The same proof also works for $\tilde{\Gamma}(T)$ since the linearity
is not used. Note that
$$||\tilde{\Gamma}(T)(\psi_1)-\tilde{\Gamma}(T)(\psi_2)||\leq
||\tilde{\Gamma}(T)(\psi_1-\psi_2)||$$ 
which is enough for proving continuity.
\end{proof}

\vspace{0.5cm}

\begin{lemma}\label{lemma2}
Let $A$ be a positivity preserving bounded linear operator and 
define on compactly supported
functions $\psi$ the following positively  homogeneous map: 
$$\hat{\Gamma}_A(T)(\psi):=\sum_{\gamma\in\mathbb{Z}^d}\chi_{\gamma,b}\;
A\left\vert T_\gamma\; 
\chi_{\gamma,b}\psi\right\vert.$$
Then $\hat{\Gamma}_A(T)$ can be extended by continuity to a bounded
map on the whole
space and $||\hat{\Gamma}_A(T)||\leq C\; ||A||\;|||T|||_\infty$. 
\end{lemma}
\begin{proof}
We note that:
$$\left \vert
  \hat{\Gamma}_A(T)(\psi_1)-\hat{\Gamma}_A(T)(\psi_2)\right\vert \leq 
\sum_{\gamma\in\mathbb{Z}^d}\chi_{\gamma,b}\;
A\big\vert T_\gamma\; 
\chi_{\gamma,b}(\psi_1-\psi_2)\big\vert=\hat{\Gamma}_A(T)(\psi_1-\psi_2)$$
due to the positivity preserving of $A$. Thus boundedness implies
continuity. But the proof of Lemma \ref{lemma1} can be repeated almost
identically, and the proof is over.
\end{proof}

\vspace{0.5cm}

\subsection{The case $\alpha>0$}

If $z\in \rho(H)$, denote by $R(z)=(H-z)^{-1}$. We construct the
operators 
$$T_\gamma(z):=e^{ib\varphi(\cdot,b^{-1/2}\gamma)}g_{\gamma,b}R(z)g_{\gamma,b}
e^{-ib\varphi(\cdot,b^{-1/2}\gamma)},\quad T(z):=\{T_\gamma(z)\}_{\gamma\in\mathbb{Z}}.$$ Then
$|||T(z)|||_\infty\leq 1/{\rm dist}(z,\sigma(H))$. Introduce the
notation 
$$fl(\x,\y,\x'):=\varphi(\x,\y)+\varphi(\y,\x')-\varphi(\x,\x').$$
The operator
$\Gamma(T(z))$ is bounded (see Lemma \ref{lemma1}). If ${\rm Id}$
denotes the identity operator, we can compute (use \eqref{kiki3}):
\begin{equation}\label{kiki8}
(H_b-z)\Gamma(T(z))={\rm Id}+S(z)
\end{equation}
where 
\begin{align}\label{kiki9}
&(S(z)\psi)(\x)\nonumber \\
&:=\sum_{\gamma\in\mathbb{Z}^d}e^{ib\varphi(\x,b^{-1/2}\gamma)}
\int_{\R^d}d\x'H(\x,\x')
\left \{e^{ibfl(\x,\x',b^{-\frac{1}{2}}\gamma)}-1\right\}g_{\gamma,b}(\x')
\left\{R(z)
  g_{\gamma,b}e^{-ib\varphi(\cdot,b^{-\frac{1}{2}}\gamma)}\psi\right\}(\x')\nonumber
\\
&+\sum_{\gamma\in\mathbb{Z}^d}e^{ib\varphi(\x,b^{-1/2}\gamma)}
\int_{\R^d}d\x'H(\x,\x')
\left\{g_{\gamma,b}(\x')-g_{\gamma,b}(\x)\right\}
\left\{R(z)
  g_{\gamma,b}e^{-ib\varphi(\cdot,b^{-\frac{1}{2}}\gamma)}\psi\right\}(\x')\nonumber
\\
&=:(S_1(z)\psi)(\x)+(S_2(z)\psi)(\x).
\end{align}
Let us analyze the contribution of the first term
$(S_1(z)\psi)(\x)$. Using the inequality (see also \eqref{iunie1})
$$\left \vert e^{ib fl(\x,\x',b^{-\frac{1}{2}}\gamma)}-1\right\vert \leq
2^{1-\epsilon}b^\epsilon|\x-\x'|^\epsilon
|\x'-b^{-\frac{1}{2}}\gamma|^\epsilon,\quad 0\leq \epsilon\leq 1,$$
we have:
\begin{align}\label{kiki10}
&|S_1(z)\psi(\x)|\\
&\leq 2^{1-\epsilon}b^\epsilon
\int_{\R^d}d\x'|H(\x,\x')|\; |\x-\x'|^\epsilon
\;\sum_{\gamma\in\mathbb{Z}^d}
g_{\gamma,b}(\x')\; |\x'-b^{-1/2}\gamma|^\epsilon
\left\vert R(z)
  g_{\gamma,b}e^{-ib\varphi(\cdot,b^{-1/2}\gamma)}\psi\right\vert
(\x')\nonumber. 
\end{align}
With the notation $L_\gamma:=g_{\gamma,b}(\cdot) 
|\cdot-b^{-1/2}\gamma|^\epsilon R(z)
  g_{\gamma,b}e^{-ib\varphi(\cdot,b^{-1/2}\gamma)}$ we see that the
  above inequality can be written as:
$$
|S_1(z)\psi(\x)|\leq 2^{1-\epsilon}b^\epsilon
\int_{\R^d}d\x'|H(\x,\x')|\; |\x-\x'|^\epsilon\left\{
\tilde{\Gamma}(L)\psi\right\} (\x').
$$
Using the fact that on the support of $g_{\gamma,b}$ we have 
$|\x'-b^{-1/2}\gamma|\leq 2b^{-1/2}$ it follows that
$|||L|||_\infty\leq Cb^{-\epsilon/2}||R(z)||$, thus: 
\begin{align}\label{kiki11}
||S_1(z)||\leq C\;\frac{b^{\epsilon/2}}
{{\rm dist}(z,\sigma(H))}\;||H||_{1,\epsilon} .
\end{align}

Let us analyze the contribution from $S_2(z)$. Using \eqref{kiki5} we
can write:
\begin{align}\label{kiki12}
&|S_2(z)\psi|(\x) \\
&\leq C\; b^{\epsilon/2}\sum_{\gamma\in\mathbb{Z}^d}
\int_{\R^d}d\x'|H(\x,\x')|\;
|\x-\x'|^\epsilon\{\chi_{\gamma,b}(\x)+
\chi_{\gamma,b}(\x')\}
\; 
\left\vert R(z)
  g_{\gamma,b}e^{-ib\varphi(\cdot,b^{-1/2}\gamma)}\psi\right\vert
(\x')\nonumber\\
&\leq C\;  b^{\epsilon/2}\sum_{\gamma\in\mathbb{Z}^d}\chi_{\gamma,b}(\x)
\int_{\R^d}d\x'|H(\x,\x')|\;
|\x-\x'|^\epsilon\left\vert R(z)
  g_{\gamma,b}e^{-ib\varphi(\cdot,b^{-1/2}\gamma)}\psi\right\vert
(\x')\nonumber \\
&+ C\;  b^{\epsilon/2}
\int_{\R^d}d\x'|H(\x,\x')|\;
|\x-\x'|^\epsilon\sum_{\gamma\in\mathbb{Z}^d}\chi_{\gamma,b}(\x')\left\vert R(z)
  g_{\gamma,b}e^{-ib\varphi(\cdot,b^{-1/2}\gamma)}\psi\right\vert
(\x').\nonumber 
\end{align}
Now denoting with $A$ the operator with integral kernel $|H(\x,\x')|\;
|\x-\x'|^\epsilon$ and with $L_\gamma=R(z)
  g_{\gamma,b}e^{-ib\varphi(\cdot,b^{-1/2}\gamma)}$ we obtain $|S_2(z)\psi|\leq C\;
b^{\epsilon/2}\left\{\hat{\Gamma}_A(L)(\psi)
  +A\tilde{\Gamma}(L)(\psi)\right\}$ thus
\begin{align}\label{kiki13}
||S_2(z)||\leq  C\; \frac{b^{\epsilon/2}}
{{\rm dist}(z,\sigma(H))}\;||H||_{1,\epsilon}.
\end{align}
Going back to \eqref{kiki8} we obtain the estimate:
\begin{align}\label{kiki14}
||S(z)||\leq  C\; \frac{b^{\epsilon/2}}
{{\rm dist}(z,\sigma(H))}\;||H||_{1,\epsilon}.
\end{align}
Now choose $0<\epsilon=\min\{\alpha,1\}$. It follows that 
$||S(z)||\leq 1/2$ for every $z$ with 
${\rm dist}(z,\sigma(H))\geq 2C\; b^{\epsilon/2}||H||_{1,\epsilon}$, and by a standard
argument it follows from \eqref{kiki8} that $z\in\rho(H_b)$. Thus for
every $x\in\sigma(H_b)$ we must have ${\rm dist}(x,\sigma(H))\leq 2C \;
b^{\epsilon/2}
||H||_{1,\epsilon}$, thus 
$$\sup_{x\in \sigma(H_b)}\inf_{y\in\sigma(H)}|x-y|\;\leq \; 2C\;
b^{\min\{\alpha/2,1/2\}}
||H||_{1,\min\{\alpha,1\}}.$$
Now we can interchange $H_b$ with $H$ because
$$H(\x,\x')=e^{-ib\phi(\x,\x')}\left
  \{e^{ib\phi(\x,\x')}H(\x,\x')\right\}=
e^{-ib\phi(\x,\x')}H_b(\x,\x')$$ and the $||\cdot||_{1,\alpha}$ norms
are invariant with respect to the multiplication with a unimodular
phase. Hence the Theorem is proved in the case $\alpha>0$. 

\vspace{0.5cm}

\subsection{The case $\alpha=0$}

Due to our uniformity condition in \eqref{chichi2} we can approximate
$H_b$ in operator norm ({\it uniformly in $b$}) 
with a sequence of operators $(H_b)_M$ which have
strong localization near their diagonal. More precisely, given
$\epsilon>0$ there exists $M=M(\epsilon)$ large enough such that
$||H_b-(H_b)_M||\leq \epsilon/3$ for every $b\in\mathbb{R}$. 
If $d(z,\sigma(H_b))>\epsilon/3$, then by writing 
$$(H_b)_M-z=[{\rm Id}-(H_b-(H_b)_M)(H_b-z)^{-1}](H_b-z)$$
it follows that $z\notin\sigma((H_b)_M)$. It means that for every
$x\in\sigma((H_b)_M)$ we must have
$d(x,\sigma(H_b))\leq\epsilon/3$. By reversing the roles of $H_b$ and
$(H_b)_M$ we conclude that $d_H(\sigma(H_b), \sigma((H_b)_M))\leq
\epsilon/3$, uniformly in $b\geq 0$. But now both $(H_b)_M$ and $H_M$
have strong localization near the diagonal, thus we can apply the
result from $\alpha>0$, obtaining a $b(\epsilon)>0$ such that for
every $|b|\leq b(\epsilon)$ we have $d_H(\sigma(H_M), \sigma((H_b)_M))\leq
\epsilon/3$. The proof is finished by the triangle inequality. 
\qed

\section{Magnetic Hamiltonians}

Let us consider in $\R^d$ a magnetic field $B$ with components of
class $BC^\infty(\R^d)$, i.e. bounded, smooth and with all its
derivatives bounded. Consider a Hamiltonian given by a 
real elliptic symbol $h$ of class $S^m_1(\R^d\times\R^d)$ with $m>0$,
i.e. $h\in C^\infty_{\text{\sf pol}}(\R^d\times\R^d)$ verifying the estimates:
$$
\forall(a,\alpha)\in\mathbb{N}^d\times\mathbb{N}^d,\ \exists C(a,\alpha)\in\R_+,\quad\underset{(x,\xi)\in\R^d\times\R^d}{\sup}<\xi>^{|\alpha|-m}\left|\left(\partial_x^a\partial_\xi^\alpha h\right)(x,\xi)\right|\leq C(a,\alpha),
$$
$$
\exists(R,C)\in\R^2_+,\quad |\xi|\geq R\Rightarrow h(x,\xi)\geq C|\xi|^m,\ \forall x\in\R^d.
$$
For our magnetic field $B$ we can choose a vector potential $A$ having
components of class $C^\infty_{\text{\sf pol}}(\R^d)$; this can always
be achieved by working with the transverse gauge:
$$
A_j(x):=-\sum\limits^d_{k=1}\int_0^1ds\,B_{jk}(sx)sx_k.
$$
Let us denote by $\mathfrak{Op}^A(h)$ the magnetic quantization of $h$ defined as in \cite{MP1}. Then, this operator is self-adjoint on the magnetic Sobolev space $H^m_A(\R^d)$ and essentially self-adjoint on the space of Schwartz test functions (see Definition 4.2 and Theorem 5.1 in \cite{IMP1}). Moreover this operator is lower semibounded and satsfies a G{\aa}rding type inequality (Theorem 5.3 in \cite{IMP1}). Thus for any $\mathfrak{z}\in\mathbb{C}\setminus[-a_0,+\infty)$, with $a_0>0$ large enough, we have that the following inverse exist
$$
\left(\mathfrak{Op}^A(h)-\mathfrak{z}\mathds{1}\right)^{-1}=\mathfrak{Op}^A\left(r^B_{\mathfrak{z}}\right)
$$
and is defined by a symbol $r^B_{\mathfrak{z}}$ of class
$S^{-m}_1(\R^d\times\R^d)$ (see Proposition 6.5 in \cite{IMP2}). But
using now Lemma A.4 in \cite{MPR1} and the fact that evidently
$S^m_1(\R^d\times\R^d)\subset S^m\big(\R^d;BC_u(\R^d)\big)$ with
$S^m\big(\R^d;BC_u(\R^d)\big)$ as in Definition A.3 of \cite {MPR1},
(or Proposition 1.3.3 of \cite{ABG}), we conclude that the symbol
$r^B_{\mathfrak{z}}$ has a partial Fourier transform (with respect to
the second variable) of class $L^1\big(\R^d;BC_u(\R^d)\big)$. In fact
looking closer to the proof of Lemma A.4 in \cite{MPR1} allows us to
conclude (see also Proposition 1.3.6. in \cite{ABG}) that the partial
Fourier transform 
$\mathfrak{F}_2^{-1}\big[r^B_{\mathfrak{z}}\big](\x,\y)$ has rapid decay in the second variable. Now, using formulas 3.28 and 3.29 in \cite{MP1}, we conclude that $\mathfrak{Op}^A\left(r^B_{\mathfrak{z}}\right)$ is an integral operator with kernel
$$
K^A(r^B_{\mathfrak{z}})(\x,\y):=\left[\tilde{\Lambda}^AS^{-1}(\mathbf{1}\otimes\mathcal{F}^{-1}_{2})r^B_{\mathfrak{z}}\right](\x,\y)
$$
with 
$$
\tilde{\Lambda}^A(\x,\y):=\exp\left\{-i\int_{\x}^{\y}A\right\},\quad 
S^{-1}(\x,\y):=\left(\frac{\x+\y}{2},\x-\y\right).
$$
In conclusion:
$$
K^A(r^B_{\mathfrak{z}})(\x,\y)=\exp\left\{-i\int_{\x}^{\y}A\right\}\left\{\left[\mathbf{1}\otimes\mathcal{F}^{-1}_{2}\right]r^B_{\mathfrak{z}}\right\}\left(\frac{\x+\y}{2},\x-\y\right).
$$
Let us also notice that if we denote by $<0,\x,\y>$ the triangle of vertices $0,\x,\y$, we have that
$$
\left|\int_{\x}^{\y}A+\int_{\y}^{\x'}A-\int_{\x}^{\x'}A\right|=\left|\int_{<\x,\y,\x'>}
B\right|\leq\|B\|_\infty|\x-\y|\;|\x'-\y|,
$$
with $\|B\|_\infty:=\underset{j,k}{\max}\underset{\x\in\R^d}{\sup}|B_{j,k}(\x)|$.

Let us conclude that for any such magnetic field and elliptic symbol
$h$, the resolvent $R=\big(\mathfrak{Op}^A(h)+a\big)^{-1}$ is a
bounded self-adjoint operator having a locally integrable integral
kernel of the form $e^{i\varphi_B(\x,\x')}T_B(\x,\x')$ with
$$
\varphi_B(\x,\x'):=-\int_{\x}^{\x'}A,\qquad T_B(\x,\x'):=
\big[S^{-1}(\mathbf{1}\otimes\mathcal{F}^{-1}_{2})r^B_{\mathfrak{z}}\big](\x,\x').
$$ 

Now let us consider a magnetic field $B_0$ with components of class
$BC^\infty(\R^d)$ and a small variation of it, in the same class,
$B_b(\x):=B_0(\x)+b\mathfrak{b}(\x)$ with $b\in[0,1]$. Given an
elliptic symbol $h$ as before we now have two Hamiltonians 
$H:=\mathfrak{Op}^{A_0}(h)$ and $H':=\mathfrak{Op}^{A}(h)$, with $A_0$
a vector potential for $B_0$ and $A$ a vector potential fos $B$. We
can write $A(\x)=A_0(\x)+b\mathfrak{a}(\x)$ with $\mathfrak{a}$ a vector
potential for $\mathfrak{b}$. Then we have the following result:
\begin{theorem}\label{teorema2}
 For $h$, $B_0$ and $B_b$ as above, consider $H=\mathfrak{Op}^{A_0}(h)$ and
 $H'=\mathfrak{Op}^{A}(h)$. For $a>0$ large enough we define the two associated resolvents as above:
$$
R:=(H+a)^{-1}=\mathfrak{Op}^{A_0}(r^{B_0}_{-a}),\ \text{with integral kernel:}\ e^{-i\left[\int_{\x}^{\x'}A_0\right]}\big[S^{-1}(\mathbf{1}\otimes\mathcal{F}^{-1}_{2})r^{B_0}_{-a}\big](\x,\x'),
$$ 
$$
R':=(H'+a)^{-1}=\mathfrak{Op}^{A_0}(r^{B_0}_{-a}),\ \text{with
  integral kernel:}\ e^{-i\left[\int_{\x}^{\x'}A\right]}\big[S^{-1}(\mathbf{1}\otimes\mathcal{F}^{-1}_{2})r^{B}_{-a}\big](\x,\x').
$$ 
 Then there exists a constant $C$ only
 depending on the symbol $h$ and on the magnetic field $B_0$ such that we have the
following estimate:
$$
d_H\big(\sigma(R),\sigma(R')\big)\leq C\sqrt{b}.
$$
\end{theorem}

\begin{proof}
Let us remark that the kernels
$S^{-1}(\mathbf{1}\otimes\mathcal{F}^{-1}_{2})r^{B_0}_{-a}$ and 
$S^{-1}(\mathbf{1}\otimes\mathcal{F}^{-1}_{2})r^{B}_{-a}$ are the
integral kernels of the operators given by the usual quantization
(without magnetic field) 
$\mathfrak{Op}$ of the symbols $r^{B_0}_{-a}$ and resp. $r^{B}_{-a}$.
\begin{proposition}\label{b-est-symb-rez}
Being symbols of negative order, both $r^{B_0}_{-a}$ and $r^{B}_{-a}$ define bounded operators on $L^2(\R^d)$ and we have that $\|\mathfrak{Op}(r^{B_0}_{-a})-\mathfrak{Op}(r^{B}_{-a})\|_{1,0}\leq Cb$.
\end{proposition}
\begin{proof}
Using the ideas and results in \cite{MP1} we shall use the magnetic Moyal composition $\sharp^B$ defined by the quantization associated to the field $B$. Let us compute (as tempered distributions):
$$
(h+a)\sharp^Br^{B_0}_{-a}-1=(h+a)\sharp^Br^{B_0}_{-a}-(h+a)\sharp^{B_0}r^{B_0}_{-a}=:s^b_{-a}.
$$
Due to the general theory developped in \cite{IMP1, IMP2} $s^b_{-a}$ is defined by a symbol of class $S^0_1(\R^d\times\R^d)$ that can be computed by the following oscillating integral:
$$
\big[(h+a)\sharp^Br^{B_0}_{-a}-1\big](\x,\xi)=(2\pi)^{-2d}\int_\Xi\int_\Xi
d\y d\eta d\x' d\zeta
$$
$$
\times
e^{-2i(<\x',\eta>-<\y,\zeta>)}\big[\omega^B(\x,\y-\x,\x'-\x)-\omega^{B_0}(\x,\y-\x,\x'-\x)\big](h+a)(\x-\y,
\xi-\eta)r^{B_0}_{-a}(\x-\x',\xi-\zeta)
$$
$$
=ib(2\pi)^{-2d}\int_\Xi\int_\Xi d\y d\eta d\x' d\zeta\
e^{-2i(<\x',\eta>-<\y,\zeta>)}\omega^{B_0}(\x,\y-\x,\x'-\x)
\theta_b(\x,\y-\x,\x'-\x)
$$
$$
\times (h+a)(\x-\y,\xi-\eta)r^{B_0}_{-a}(\x-\x',\xi-\zeta),
$$
where 
$$
\theta_b(\x,\y-\x,\x'-\x)=e^{-ib\int_{<\x+\y-\x',\x+\x'-\y,\x-\y-\x'>}\mathfrak{b}}-1
$$
is a function of class $BC^\infty\big(\R^d;C^\infty_{\text{\sf pol}}(\R^d\times\R^d)\big)$ and we have the following estimates for its derivatives:
$$
\left|\big(\partial_{\x}^\rho\partial_{\y}^\mu\partial_{\x'}^\nu\theta_b\big)(\x,\y-\x,\x'-\x)\right|\leq 
C_{\rho,\mu,\nu}b^{1+|\rho+\mu+\nu}|\y|^{|\mu|}|\x'|^{|\nu|}.
$$
Now using Proposition 8.45 in \cite{IMP2} we conclude that $(h+a)\sharp^Br^{B_0}_{-a}-1$ is a symbol of type $S^0_1(\R^d\times\R^d)$ with seminorms of order at least $b$ and using Remark 3.3 in \cite{IMP1} we conclude that it defines a bounded operator with norm of order $b$. Thus for $b$ small enough we can invert $1+s^b_{-a}$ and obtain that (using once again Proposition 8.45 in \cite{IMP2} and the Calderon-Vaillancourt Theorem \cite{Fo})
$$
r^{B}_{-a}=r^{B_0}_{-a}\sharp^B\left\{1+s^b_{-a}\right\}^{-_B},\quad r^{B}_{-a}-r^{B_0}_{-a}=-r^{B}_{-a}\sharp^Bs^b_{-a},
$$
$$
\|\mathfrak{Op}(r^{B_0}_{-a})-\mathfrak{Op}(r^{B}_{-a})\|_{1,0}\leq Cb.
$$
\end{proof}

\vspace{0.5cm}

Now we shall consider the bounded self-adjoint operators $R_b$ with
the kernel 
$e^{-i\left[\int_{\x}^{\x'}A\right]}\big[S^{-1}(\mathbf{1}\otimes\mathcal{F}^{-1}_{2})r^{B_0}_{-a}\big](\x,\x')$. 
Due to the above Proposition, by replacing $R'$ with $R_b$ we make an
error of order $b$ in operator norm on $L^2(\R^d)$. Now we see that
$R_b$ is a Harper-type family, for which we can apply the results of
Section 2. Here, $R=R_0$. We note that the integral kernels of $R_b$
have a common factor independent of $b$ which is of class $C_{1,\alpha}$ for any
$\alpha\geq0$. Moreover, the integral kernels of $R_b$ only differ by a unimodular exponential factor 
$e^{ib\varphi(\x,\x')}$ where
$\varphi(\x,\x'):=-\int_{\x}^{\x'}\mathfrak{a}$ satisfies
\eqref{iunie1}. 

Therefore Theorem \ref{teorema1} implies that
$d_H(\sigma(R_b),\sigma(R))\leq C\; \sqrt{b}$, and since
$d_H(\sigma(R_b),\sigma(R'))\leq C\; b$ it follows that 
\begin{equation}\label{iunie28}
d_H(\sigma(R'),\sigma(R))\leq C\; \sqrt{b},
\end{equation}
which finishes the proof of the theorem.
\end{proof}

\vspace{0.5cm}

\noindent {\bf Acknowledgments.}  H.C. acknowledges support from the Danish F.N.U. grant Mathematical
Physics, and thanks Gheorghe Nenciu for many fruitful discussions. R.P.
acknowledges support
from CNCSIS grant PCCE 8/2010 {\it Sisteme
diferentiale in analiza neliniara si aplicatii}, 
and thanks Aalborg University for hospitality.


\begin{thebibliography}{99}

\bibitem{ABG} W.O. Amrein, A. Boutet de Monvel and V. Georgescu:
{\it $C_0$-Groups, Commutator Methods and Spectral Theory of N-Body
Hamiltonians}, Birkh\"auser Verlag, 1996.

\bibitem{AMP} N. Athmouni, M. M\u{a}ntoiu, and R. Purice: On the continuity of
spectra for families of magnetic pseudodifferential operators. Journal of
Mathematical Physics 51, 083517 (2010); doi:10.1063/1.3470118 (15 pages).

\bibitem{AS} Avron, J.E., Simon, B.: Stability of gaps for periodic potentials under variation of a magnetic
field. J. Phys. A: Math. Gen. {\bf 18}, 2199-2205 (1985)

\bibitem{AMS} Avron, J., van Mouche, P.H.M., Simon, B.: On the measure of the spectrum for the almost
Mathieu operator. Commun. Math. Phys. {\bf 132}, 103-118, (1990). 
Erratum in Commun. Math.
Phys. {\bf 139}, 215 (1991)

\bibitem{Bell} Bellissard, J.: Lipshitz Continuity of Gap Boundaries
for Hofstadter-like Spectra. Commun. Math. Phys. {\bf 160}, 599-613 (1994)

\bibitem{BC} Briet, P., Cornean, H.D.: 
Locating the spectrum for magnetic Schr\"odinger and Dirac operators.  
{\it Comm. Partial Differential Equations}  {\bf 27}  no. 5-6, 1079--1101 
(2002)

\bibitem{BEY} Choi, M.D., Elliott, G.A., Yui, N.: 
Gauss polynomials and the rotation algebra. Invent. Math.
{\bf 99}, 225-246 (1990)

\bibitem{Cornean} Cornean, H.D.: 
\newblock On the Lipschitz continuity of spectral bands of Harper-like
and magnetic 
Schr\"odinger operators.
\newblock  Ann. Henri Poincar{\'e}  {\bf 11}, 973-990  (2010)

\bibitem{Ell} Elliott, G.: Gaps in the spectrum of an almost periodic 
Schrodinger operator. C.R. Math. Rep.
Acad. Sci. Canada {\bf 4}, 255-259 (1982)

\bibitem{Fo} Folland, B.G.: Harmonic Analysis in Phase Space. Annals of Mathematics Studies. Princeton University Press, 1989.

\bibitem{HR}Haagerup, U., R{\o }rdam, M.: 
Perturbations of the rotation $C^*$-algebras and of the
Heisenberg commutation relation. Duke Math. J. {\bf 77}, 627-656 (1995)

\bibitem{HKS} Helffer, B., Kerdelhue, P., Sj{\"o}strand, J.:
  M{\'e}moires de la SMF, S{\'e}rie 2 {\bf 43}, 1-87 (1990)

\bibitem{He-Sj1} Helffer, B., Sj{\"o}strand, J.: 
\newblock Equation de Schr\"odinger avec champ magn{\'e}tique et {\'e}quation de Harper.
\newblock  Springer Lecture Notes in Phys.  {\bf 345}, 118-197  (1989)

\bibitem{He-Sj2} Helffer, B., Sj{\"o}strand, J.: Analyse
  semi-classique pour l'{\'e}quation de Harper. II. Bull. Soc. Math.
France {\bf 117}, Fasc. 4, Memoire 40 (1990)

\bibitem{HJ} Herrmann, D.J.L., Janssen, T.: 
On spectral properties of Harper-like models. 
J. Math. Phys. {\bf 40} (3), 1197 (1999)

\bibitem{If} V. Iftimie: {\it Op\'erateurs differentiels magn\'etiques: Stabilit\'e des trous dans le spectre,
invariance du spectre essentiel et applications}, Commun. in P.D.E. {\bf 18}, 651-686, (1993).

\bibitem{IMP1} Iftimie, V., M\u antoiu, M., Purice, R.: 
Magnetic Pseudodifferential Operators.  Publications of the Research Institute for Mathematical Sciences 
{\bf 43} (3), 585-623 (2007)

\bibitem{IMP2}Iftimie, V., M\u antoiu, M., Purice, R.: 
Commutator Criteria for Magnetic Pseudodifferential Operators.  Communications in Partial Differential Equations 
\textbf{35}, 1058—1094, (2010).
\bibitem{Ko} Kotani, M.: 
Lipschitz continuity of the spectra of the magnetic transition
operators on a crystal lattice.  J. Geom. Phys.  {\bf 47} (2-3), 323--342 (2003)

\bibitem{KO} Karasev, M.V. and Osborn, T.A.: {\it Symplectic areas,
quantization and dynamics in electromagnetic fields}, J. Math. Phys. {\bf
43} (2), 756--788, 2002.

\bibitem{LMS} Lein, M., M\u antoiu, M., Richard, S.: Magnetic
  pseudodifferential operators with coefficients in $C^*$-algebras. 
http://arxiv.org/abs/0901.3704v1 (2009)

\bibitem{MP1}M\u antoiu, M., Purice, R.: The magnetic Weyl calculus. 
 J. Math. Phys. {\bf 45} (4), 1394--1417 (2004)

\bibitem{MPR1} M\u antoiu, M., Purice, R., Richard, S.: Spectral and propagation results for magnetic 
Schrodinger operators; 
A $C^*$-algebraic framework. {\it J. Funct. Anal.} {\bf 250} (1), 42-67 (2007)

\bibitem{MPR2} Mantoiu, M.; Purice, R.; Richard, S.: Twisted crossed products and
magnetic pseudodifferential operators. Advances in operator algebras and
mathematical physics, 137--172, Theta Ser. Adv. Math., 5, Theta, Bucharest,
2005.

\bibitem{Nen1}
 Nenciu, G.: On asymptotic perturbation theory for quantum mechanics: 
almost invariant subspaces and gauge invariant magnetic perturbation
theory.  J. Math. Phys. {\bf 43} (3), 1273--1298 (2002)

\bibitem{Nen2}Nenciu, G.: Stability of energy gaps under variation of 
the magnetic field. Lett. Math. Phys. {\bf 11},
127-132 (1986)

\bibitem{Nen3} Nenciu, G.:
On the smoothness of gap boundaries for generalized Harper operators. 
Advances in operator algebras and mathematical physics, 
Theta Ser. Adv. Math. {\bf 5}, 173-182, Theta, Bucharest, 2005. 
arXiv:math-ph/0309009v2




\end{thebibliography}
\end{document}